\documentclass[10pt]{article}
\usepackage{amsmath, amsthm, amssymb, amsfonts, mathrsfs, mathtools, upgreek}
\usepackage{graphicx}
\usepackage{makeidx}
\usepackage[left=2cm,right=2cm,top=2cm,bottom=2cm]{geometry}
\usepackage{caption}
\usepackage{subcaption}
\usepackage[section]{placeins}
\usepackage{enumitem}
\usepackage{tikz}
\usepackage{stmaryrd}
\usepackage{hyperref}
\usepackage{upquote}

\usepackage{hyperref}

\hypersetup{colorlinks=true, citecolor={blue}, linkcolor=blue, filecolor=magenta, urlcolor=cyan}
\newtheorem{theorem}{Theorem}[section]

\newtheorem{lema}[theorem]{Lemma}

\newtheorem{defi}[theorem]{Definition}
\newtheorem{assum}[theorem]{Assumption}

\numberwithin{equation}{section}

\newcommand{\ip}[2]{\langle #1, #2 \rangle}

\newcommand{\aut}{\operatorname{Aut}}

\newcommand{\cay}{\operatorname{Cay}}
\newcommand{\spec}{\operatorname{Spec}}

\def \Zl {{\mathbb F}}
\def \Zl {{\mathbb Z}}
\def \Nl {{\mathbb N}}
\def \Rl {{\mathbb R}}
\def \Zl {{\mathbb Z}}
\def \Ql {{\mathbb Q}}
\def \Cl {{\mathbb C}}

\def \ld {{\lambda}}

\def \eu {{\textbf{e}_u}}
\def \ev {{\textbf{e}_v}}
\def \ex {{\textbf{e}_x}}
\def \ey {{\textbf{e}_y}}
\def \ez {{\textbf{e}_z}}
\def \ew {{\textbf{e}_w}}
\def \Gr {\mathcal{G}_{R}}
\def \Gn {\mathcal{G}_{\Zl_n}}
\def \Zn {\mathbb{Z}_n}
\def \congp {{\simeq_{\textnormal{P}}}}
	\title{State transfer in Grover walks on unitary and quadratic unitary Cayley graphs over finite commutative rings}
	
	\author{ Koushik Bhakta and Bikash Bhattacharjya\\
		Department of Mathematics\\
		Indian Institute of Technology Guwahati, India\\
		b.koushik@iitg.ac.in, b.bikash@iitg.ac.in }
	
	\date{}
	
\begin{document}
		\maketitle
		
		\vspace{-0.3in}
		
		\begin{center}{\textbf{Abstract}}\end{center}
		\noindent This paper focuses on periodicity and perfect state transfer in Grover walks on two well-known families of Cayley graphs, namely, the unitary Cayley graphs and the quadratic unitary Cayley graphs. Let $R$ be a finite commutative ring. The unitary Cayley graph $G_R$ has vertex set $R$, where two vertices $u$ and $v$ are adjacent if $u-v$ is a unit in $R$. We provide a necessary and sufficient condition for the periodicity of the Cayley graph $G_R$. We also completely determine the rings $R$ for which $G_R$ exhibits perfect state transfer. The quadratic unitary Cayley graph $\mathcal{G}_R$ has vertex set $R$, where two vertices $u$ and $v$ are adjacent if $u-v$ or $v-u$ is a square of some units in $R$. It is well known that any finite commutative ring $R$ can be expressed as $R_1\times\cdots\times R_s$, where each $R_i$ is a local ring with maximal ideal $M_i$ for $i\in\{1,\hdots,s\}$. We characterize periodicity and perfect state transfer on $\mathcal{G}_R$ under the condition that $|R_i|/|M_i|\equiv 1 \pmod 4$ for $i\in\{1,\hdots,s\}$. Also, we characterize periodicity and perfect state transfer on $\Gr$, where $R$ can be expressed as $R_0\times\cdots\times R_s$ such that $|R_0|/|M_0|\equiv3\pmod 4$, and $|R_i|/|M_i|\equiv1\pmod4$ for $i\in\{1,\hdots, s\}$, where $R_i$ is a local ring with maximal ideal $M_i$ for $i\in\{0,\hdots,s\}$.
		
		\vspace*{0.3cm}
		
		\section{Introduction}
		Quantum walks~\cite{kendon} on graphs are the quantum counterparts of classical random walks and provide a framework for modeling the movement of quantum particles on graphs. The study of quantum walks is significant due to their numerous applications, including quantum search algorithms~\cite{search}, quantum algorithm design~\cite{ambain}, graph isomorphism testing~\cite{graphiso}, and others. Quantum walks are primarily categorized into two types: continuous-time quantum walks~\cite{state} and discrete-time quantum walks~\cite{chen}. In continuous-time quantum walks, the evolution is typically governed by the adjacency matrix of a graph, allowing specific properties of the walks to directly correspond to properties of the graph (see~\cite{pal}). In contrast, the transition matrix of a discrete-time quantum walk is formed as the product of two non-commuting sparse unitary matrices.
		Furthermore, the selection of these sparse matrices is not unique (see~\cite{zhan2}).	This gives rise to different models of discrete-time quantum walks. This paper focuses on Grover walks, a type of discrete-time quantum walk. We refer the reader to~\cite{circ, spec} for more information about Grover walks on general graphs. This paper focuses on two types of state transfer properties in Grover walks: periodicity~\cite{bipartite} and perfect state transfer~\cite{dqw3}. Periodicity in a discrete-time quantum walk refers to the phenomenon where the state of the quantum walk returns to its initial state after a specific number of time steps. Perfect state transfer refers to the complete and deterministic transfer of a quantum state from one position to another after a specific number of time steps.
		
		The study of state transfer in Grover walks has received much attention in the past decade. Higuchi et al.~\cite{hig1} studied the periodicity of Grover walks on complete graphs, complete bipartite graphs and strongly regular graphs. Kubota et al.~\cite{bethetrees} characterize the periodicity of Grover walks on generalized Bethe trees. Ito et al. in~\cite{completegraph} proved that every complete graph with a self-loop on each vertex is periodic. Yoshie~\cite{distance} studied the periodicity of Grover walks on distance-regular graphs. Kubota et al.~\cite{mixedpaths} studied the periodicity of Grover walks on mixed paths and mixed cycles. Some additional studies on the periodicity of Grover walks can be found in~\cite{kubota, qq, regluar, oddperiodic}. Zhan \cite{dqw1} gave an infinite family of graphs exhibiting perfect state transfer. Kubota and Segawa~\cite{pstdc} studied perfect state transfer in Grover walks between vertex-type states. Kubota and Yoshino characterized the perfect state transfer in Grover walks on circulant graphs with valency up to $4$ in~\cite{circ}.

		This paper aims to study periodicity and perfect state transfer in Grover walks on unitary and quadratic unitary Cayley graphs over finite commutative rings. A commutative ring is called a \emph{local ring} if it has a unique maximal ideal. It is well known~\cite{dummit} that every finite commutative ring can be represented as a direct product of finite local rings, and this decomposition is unique up to the permutation of the local rings. 
		
		\begin{assum}\label{as}
			{\em	Let $R$ be a finite commutative ring with unity $1$ $(\neq 0)$. Let the decomposition of $R$ be $R_1\times\cdots \times R_s,$ where $s$ is a positive integer, and $R_i$ is a local ring with unique maximal ideal $M_i$ for $1\leq i\leq s$. Let $m_i=|M_i|$ for $i\in\{1,\hdots,s\}$ and $m=m_1\cdots m_s$. We also assume that 
				$$\dfrac{|R_1|}{m_1}\geq \cdots \geq\dfrac{|R_s|}{m_s}.$$}
		\end{assum}
		
		We now discuss some classifications of finite rings. We refer to \cite{clring} for more details on the classification of finite rings. For an additive abelian group $\Gamma$, let $\Gamma(0)$ be the ring $(\Gamma, +,\cdot)$, in which $+$ is the addition in the group $\Gamma$, and $a\cdot b=0$ for all $a,b\in \Gamma$.  Let $C_m$ denote the cyclic additive group of order $m$. The ring $\Zl_n$ denotes the ring of integers modulo $n$.
		
		The classification of finite rings of order $p$, $p^2$ and $pq$, where $p$ and $q$ are distinct primes, can be found in \cite{clring}.	If $p$ is a prime number, then there are exactly two rings of order $p$, namely, $\Zl_p$ and $C_p(0)$, up to isomorphism. If $p$ and $q$ are distinct primes, then there are exactly four rings of order $pq$, up to isomorphism. These are $\Zl_{pq}$, $C_{pq}(0)$, $C_p(0)\times \Zl_q$ and $\Zl_p\times C_q(0)$. 
		
		If $R$ is a finite ring, then its additive group is a finite abelian group, and so it has finitely many generators, say, $a_1,\hdots,a_k$ of order $n_1,\hdots,n_k$, respectively. As $a_ia_j$ is in an element in $R$, we find that $a_ia_j=\sum_{t=1}^{k}c_{ij}^ta_t$ with $c_{ij}^t\in\Zl_{n_t}$. Thus, 
		$$R=\left\langle a_1,\hdots,a_k:n_ia_i=0~\text{for}~ i\in\{1,\hdots,k\}, a_ia_j=\sum_{t=1}^{k}c_{ij}^ta_t\right\rangle.$$
		The following theorem completely classifies finite rings of order $p^2$, where $p$ is a prime number.
		\begin{theorem}[\cite{clring}]\label{clr}
			For any prime number $p$, there are exactly 11 rings of order $p^2$, namely:
			\begin{align*}
				\mathscr{A}(p) & = \mathbb{Z}_{p^2}, \\
				\mathscr{B}(p) & = \langle a : p^2a = 0, \, a^2 = pa \rangle, \\
				\mathscr{C}(p) & = C_{p^2}(0), \\
				\mathscr{D}(p) & = \mathbb{Z}_p \times \mathbb{Z}_p, \\
				\mathscr{E}(p) & = \langle a, b : pa = pb = 0, \, a^2 = a, \, b^2 = b, \, ab = a, \, ba = b \rangle, \\
				\mathscr{F}(p) & = \langle a, b : pa = pb = 0, \, a^2 = a, \, b^2 = b, \, ab = b, \, ba = a \rangle, \\
				\mathscr{G}(p) & = \langle a, b : pa = pb = 0, \, a^2 = 0, \, b^2 = b, \, ab = a, \, ba = a \rangle, \\
				\mathscr{H}(p) & = \mathbb{Z}_p \times C_p(0), \\
				\mathscr{I}(p) & = \langle a, b : pa = pb = 0, \, a^2 = b, \, ab = 0 \rangle, \\
				\mathscr{J}(p) & = (C_p \times C_p)(0),~\text{and} \\
				\mathscr{K}(p) & = \mathbb{F}_{p^2},~\text{where $\mathbb{F}_n$ denotes the finite field of order $n$}.
			\end{align*}
		\end{theorem}
		Note that among these $11$ rings of order $p^2$, only $4$ have a multiplicative identity. These are $\mathscr{A}(p)$, $\mathscr{D}(p)$, $\mathscr{G}(p)$ and $\mathscr{K}(p)$. For example, $\mathscr{G}(2)=\{0,a,b,a+b\}$ has the multiplicative identity $b$. Observe that $\mathscr{G}(2)$ is isomorphic to the quotient ring $\mathbb{F}_2[x]/(x^2)$.
		
		Let $\Gamma$ be a finite additive group. Let $C\subset\Gamma\setminus\{0\}$ with $-C=C$, where $0$ is the identity element of $\Gamma$ and $-C=\{-c:c\in C\}$. The \emph{Cayley graph} of $\Gamma$ with respect to the connection set $C$, denoted $\cay(\Gamma, C)$, is a simple graph whose vertex set is $\Gamma$, and two vertices $u$ and $v$ are adjacent if $u-v\in C$. 
		
		Let $R$ be a finite commutative ring, and $R^\times$ denote the set of all units of $R$. Let $Q_R=\{u^2:u\in R^\times\}$ and $T_R=Q_R\cup (-Q_R)$. The Cayley graphs $\cay(R,R^\times)$ and $\cay(R,T_R)$ are known as the unitary Cayley graph and quadratic unitary Cayley graph over $R$, respectively. For the sake of simplicity, we denote $\cay(R,R^\times)$ by $G_R$ and $\cay(R,T_R)$ by $\Gr$. We refer to~\cite{ucr2, qucfirst, quc, ucr3, ucg} for more details about the graphs $G_R$ and $\Gr$.

		Thongsomnuk and Meemark~\cite{ctw_uc} studied perfect state transfer in continuous-time quantum walk on unitary Cayley graphs and gcd graphs. We studied periodicity and perfect state transfer in Grover walks on $G_{\mathcal{\Zl}_n}$ and $\Gn$ in~\cite{bhakta1, bhakta2}. In this paper, we extend the main results in \cite{bhakta1, bhakta2} over any finite commutative ring with unity. The periodicity and perfect state transfer on the unitary Cayley graph $G_R$ are discussed in Theorems \ref{ucrp} and \ref{ucrpst}. Also, the periodicity and perfect state transfer on the quadratic unitary Cayley graph $\Gr$ are explored in Theorems~\ref{qucrp}, \ref{qucrpst}, \ref{peroth} and \ref{lastth}, under certain conditions on $R$.

		The rest of the paper is organized as follows. In the next section, we give the definition of Grover walks on graphs and introduce the concepts of periodicity and perfect state transfer in Grover walks. Additionally, we present some results related to the periodicity and perfect state transfer. The third section explores the relationship between graph automorphisms and perfect state transfer. We also show that graph isomorphism preserves perfect state transfer. The fourth section provides a characterization of periodicity and perfect state transfer on unitary Cayley graphs over finite commutative rings with unity. In the last section, we characterize periodicity and perfect state transfer on quadratic unitary Cayley graphs over finite commutative rings with some additional conditions.
		\section{Grover walks}
		Let $G$ be a finite, simple, and connected graph with vertex set $V(G)$ and edge set $E(G)$. If two vertices $u$ and $v$ in $G$  are adjacent, then the edge between them is denoted by $uv$. Note that $uv$ and $vu$ represent the same edge. Thus, $E(G)\subseteq\{uv:u,v\in V(G), u\neq v\}$, with the convention that $uv=vu$.  Define $\mathcal{A}(G)=\{(u,v),(v,u): uv\in E(G)\},$ the set of symmetric arcs of $G$. Let $a=(u,v)\in\mathcal{A}(G)$.  The \emph{tail} $u$ and \emph{head} $v$ of $a$ are denoted by $o(a), t(a)$, respectively. The inverse arc of $a$, denoted $a^{-1}$, is the arc $(v,u)$.

		For a matrix $M$, let $M_{uv}$ denote the entry in the $u$-th row and $v$-th column of $M$. The \emph{boundary matrix} $N:=N(G)\in \mathbb{C}^{V(G)\times \mathcal{A}(G)}$ of $G$ is defined by $N_{ua}=\frac{1}{\sqrt{\deg u}}\delta _{u, t(a)},$ where $\delta_{a,b}$ is the Kronecker delta function and $\deg u$ is the degree of the vertex $u$. The \emph{shift matrix} $S:=S(G)\in \mathbb{C}^{\mathcal{A}(G) \times \mathcal{A}(G)}$ of $G$ is defined by $S_{ab}=\delta_{a,b^{-1}}.$ The \emph{time evolution matrix} $U:=U(G)\in \mathbb{C}^{\mathcal{A}(G)\times \mathcal{A}(G)}$ of $G$ is defined by 
		$$U=S(2N^*N-I).$$ 
		We refer to \cite{mixedpaths} for more details about these matrices. A discrete-time quantum walk on a graph $G$ is determined by a unitary matrix that acts on complex functions defined on the set of symmetric arcs of $G$. The discrete-time quantum walks defined by $U$ are referred to as \emph{Grover walks}. For $\Phi,\Uppsi\in \Cl^{\mathcal{A}(G)}$, let $\ip{\Phi}{\Uppsi}$ denote the Euclidean inner product of $\Phi$ and $\Uppsi$. A vector $\Phi\in \Cl^{\mathcal{A}(G)}$ is called a \emph{state} if $\ip{\Phi}{\Phi}=1$. Let $\Phi_0$ be the initial quantum state. Then the quantum state $\Phi_\tau$ at time $\tau$ is given by the evolution equation $\Phi_\tau=U^\tau\Phi_0$. 
		\begin{defi}  
			{\em The Grover walk on a graph $G$ is \emph{periodic} if there exists a positive integer $\tau$ such that the time  evolution operator $U$ of $G$ satisfies $U^\tau = I$. }
		\end{defi}  
		For convenience, we say  a graph $G$ is periodic to mean the Grover walk on $G$ is periodic. The \emph{discriminant matrix} $P:=P(G)\in \Cl^{V(G)\times V(G)}$ of a graph $G$ is defined by $P=NSN^*$. The \emph{adjacency matrix} $A:=A(G)\in \Cl^{V(G)\times V(G)}$ of $G$ is defined by 
		$$A_{uv} = \left\{ \begin{array}{rl}
			1 &\mbox{ if }
			uv\in E(G) \\ 
			0 &\textnormal{ otherwise.}
		\end{array}\right.$$ 
		If a graph is regular, then the discriminant and adjacency matrices are correlated.
		\begin{lema}[\cite{qq}]\label{discp}
			Let $P$ and $A$ be the discriminant and adjacency matrix of a graph $G$, respectively. If $G$ is a $k$-regular graph, then $P=\frac{1}{k}A$.
		\end{lema}
		Define $\Delta=\{a\pm\sqrt{b}:a,b\in \Ql~\text{and}~b ~\text{is not a square}\}$ and $\overline{\Delta}=\Rl\setminus(\Ql\cup\Delta)$. Let $\spec_P(G)$ denote the set of all distinct eigenvalues of $P$ to a graph $G$. Define $\spec_P^F(G)=F\cap\spec_P(G)$, for a subset $F$ of real numbers. Let $\Re=\{(z+\overline{z})/2: z\in \Cl~\text{and}~ z^n=1~\text{for some positive integer $n$}\}$. We use the spectral analysis of the discriminant matrix $P$ to investigate the periodicity of a graph.
		\begin{theorem}[\cite{bhakta2}]\label{ls}
			Let $G$ be a regular graph with discriminant matrix $P$. Then $G$ is periodic if and only if $$\spec_P^\Ql(G)\subseteq\left\{\pm1,\pm\frac{1}{2},0\right\},~\spec_P^\Delta(G)\subseteq\left\{\pm \frac{\sqrt{3}}{2}, \pm\frac{1}{4}\pm\frac{\sqrt{5}}{4},\pm\frac{1}{\sqrt{2}}\right\}~\text{and}~\spec_P^{\overline{\Delta}}(G)\subset \Re.$$
		\end{theorem}
		A graph $G$ is called an \emph{integral graph} if $\spec_A(G)\subset \Zl$. By the previous result, an integral regular graph $G$ is periodic if and only if $\spec_P(G)\subseteq \left\{\pm1,\pm\frac{1}{2},0\right\}$.
		
		Let $\Phi$ and $\Uppsi$ be two distinct states of a graph $G$. Then we say \emph{perfect state transfer} occurs from $\Phi$ to $\Uppsi$ at time $\tau$ if there exists a unimodular complex number $\gamma$ such that $U^\tau\Phi=\gamma\Uppsi$. The following lemma gives an equivalent definition of perfect state transfer.
		\begin{lema} [\cite{dqw1}]\label{st11}
			Let $G$ be a graph and $U$ be its time evolution matrix. The perfect state transfer occurs from a state $\Phi$ to another state $\Psi$ at time $\tau$ if and only if  $|\ip{U^\tau \Phi}{\Uppsi}|=1.$ 
		\end{lema}
		We focus on transferring states that are concentrated on specific vertices of a graph, referred to as vertex-type states. A state $\Phi$ is called a \emph{vertex-type} state if there exists $u\in V(G)$ such that $\Phi=N^*\eu$. See \cite{pstdc} for motivation on focusing on the transfer of states between vertex-type states of a graph.
		
		\begin{defi}
			{\em A graph $G$ exhibits \emph{perfect state transfer}  from a vertex $u$ to another vertex $v$ at time $\tau\in\Nl$ if $G$ exhibits perfect state transfer from the state $N^*\eu$ to the state $N^*\ev$ at time $\tau$, that is, there exists a unimodular complex number $\gamma$ such that $U^\tau N^*\eu=\gamma N^*\ev$.}
		\end{defi}
		We say a graph $G$ exhibits perfect state transfer if there exist vertices $u$ and $v$ in $G$ such that perfect state transfer occurs from $u$ to $v$ at some time $\tau$.

		A graph is called \emph{vertex-transitive} if, for any pair of vertices in $G$, there exists an automorphism that maps one vertex to the other. It is known that periodicity is a necessary condition for the occurrence of perfect state transfer on a vertex-transitive graph. 
		\begin{theorem}[\cite{bhakta1}]\label{thmper}
			Let $G$ be a vertex-transitive graph. If $G$ exhibits perfect state transfer, then it is periodic.
		\end{theorem}
		Kubota and Segawa \cite{pstdc} established a connection between the Chebyshev polynomials and the occurrence of perfect state transfer in a graph. The following recurrence relation defines the Chebyshev polynomials of the first kind: 
		$$T_0(x)=1,~T_1(x)=x~\text{and}~T_{n+1}(x)=2xT_n(x)-T_{n-1}(x)~\text{for}~n\geq1.$$

		\begin{lema}[\cite{pstdc}]\label{ch11}
			Let $T_n(x)$ be the Chebyshev polynomial of the first kind.	Let $G$ be a graph with the time evolution matrix $U$ and discriminant $P$. Then $NU^\tau N^* = T_\tau (P)$ for $\tau \in \Nl \cup \{0\}$. 
		\end{lema}
		Note that the discriminant $P$ of a graph is a symmetric matrix. Thus, it has a spectral decomposition. Let $\mu_1, \hdots, \mu_d$ be the distinct eigenvalues of $P$, and the idempotent projections onto the corresponding eigenspaces be $E_1,\hdots,E_d$, respectively. Then the spectral decomposition of $P$ is $P=\sum_{r=1}^{d} \mu_r E_r.$ It is known that $E_r^2=E_r,~ E_rE_s=0,~ E_r^t=E_r$ and $E_1+\cdots + E_d=I$ for $1\leq r,s \leq d$ and $r\neq s$. Because of the preceding properties, if $h(x)$ is a polynomial, then we have 
		\begin{equation}\label{sd}
			h(P)=\sum_{r=1}^{d} h(\mu_r) E_r.
		\end{equation}
		Let $u$ be a vertex of a graph $G$. The \emph{eigenvalue support} of $u$ with respect to $P$, denoted $\Theta_P(u)$, is defined as $\Theta_P(u)=\{\mu_r\in \spec_P(G): E_r \eu \neq 0\}.$ The following theorem provides a necessary condition on the eigenvalues of $P$ for the occurrence of perfect state transfer in a graph.
		\begin{theorem}[\cite{pstdc}]\label{p1}
			Let  $u$ and $v$ be two distinct vertices of a graph $G$, and let $P$ be its discriminant. If perfect state transfer occurs from $u$ to $v$ at time $\tau$, then $T_\tau(\mu)=\pm 1$ for all $\mu \in \Theta_P(u)$.
		\end{theorem}
		\section{Graph symmetry and perfect state transfer}
		This section explores the relationship between graph automorphisms and perfect state transfer. Additionally, we examine how isomorphisms preserve the occurrence of perfect state transfer. First, we present an equivalent definition of perfect state transfer and a necessary condition for its occurrence on a graph.
		
		\begin{lema}\label{pstdef}
			Let $G$ be a graph and $P$ be its discriminant. Then perfect state transfer occurs from a vertex $u$ to vertex $v$ on $G$ at time $\tau$ if and only if there exists $\gamma\in \{-1, 1\}$ such that $T_\tau(P)\eu=\gamma\ev$.
		\end{lema}
		\begin{proof}
			Perfect state transfer occurs from a vertex $u$ to vertex $v$ in $G$ at time $\tau$ if and only if
			\begin{align}
				1&=|\ip{U^\tau N^*\eu}{N^* \ev}| \nonumber \tag*{(by Lemma \ref{st11})}\\
				&=|\ip{NU^\tau N^*\eu}{\ev}| \nonumber\\
				&=|\ip{T_{\tau}(P)\eu}{\ev}| \nonumber\tag*{(by Lemma \ref{ch11})}\\
				&= |T_{\tau}(P)_{uv}|. \nonumber
			\end{align}
			Since $T_\tau(P)$ is a real matrix, we have $T_\tau(P)_{uv}\in\{-1,1\}$.
		\end{proof}
		Since $P$ is a symmetric matrix, $T_\tau(P)$ is also a symmetric matrix. Therefore if $T_\tau(P)\eu=\gamma\ev$, then $T_\tau(P)\ev=\gamma\eu$. Consequently, if perfect state transfer occurs from vertex $u$ to another vertex $v$ at time $\tau$ in $G$, then perfect state transfer occurs from $v$ to $u$ at the same time $\tau$. Thus, instead of stating that perfect state transfer occurs from vertex $u$ to vertex $v$, we say that perfect state transfer occurs between  $u$ and $v$.
		\begin{lema}\label{pstc}
			Let $P$ be the discriminant of a graph $G$ with spectral decomposition $P=\sum_{r=1}^{d}\mu_r E_r$. If perfect state transfer occurs between vertices $u$ and $v$, then $E_r\eu=\pm E_r\ev$ for each $r$. 
		\end{lema}
		\begin{proof}
			If perfect state transfer occurs between $u$ and $v$, then there exists $\tau\in \Nl$ and $\gamma\in\{-1,1\}$ such that
			$T_\tau(P)\eu=\gamma\ev.$
			From \eqref{sd}, we have
			$$\sum_{r=1}^{d} T_\tau(\mu_r)E_r\eu=\gamma\ev.$$
			Therefore $$T_\tau(\mu_r)E_r\eu=\gamma E_r\ev~\text{for each $r\in\{1,\hdots,d\}$}.$$
			Hence by Theorem \ref{p1}, $E_r\eu=\pm E_r\ev$ for $\mu_r\in \Theta_P(u)$. For $\mu_r\notin\Theta_P(u)$, $E_r\eu=0=E_r\ev$. Thus, we have the result.
		\end{proof}
		
		Let $\aut(G)$ be the set of all automorphisms of a graph $G$. For any $\sigma\in \aut(G)$, note that $\sigma$ corresponds to a permutation matrix $M_\sigma\in\Cl^{V(G)\times V(G)}$,
		given by $(M_\sigma)_{uv}=\delta_{u,\sigma(v)}$. For $u\in V(G)$, let $\aut(G)_u=\{\sigma\in\aut(G):\sigma(u)=u\}$.

		\begin{lema}\label{per1}
			Let $G$ be a graph with discriminant $P$. Let $M_\sigma$ be the permutation matrix of an automorphism $\sigma$ of $G$. Then $M_\sigma P=PM_\sigma$. 
		\end{lema}
		\begin{proof}
			Let $\sigma\in\aut(G)$ and $M_\sigma$ be the corresponding permutation matrix. Let $u,v\in V(G)$. We show that $(M_\sigma P)_{uv}=(PM_\sigma)_{uv}.$ Let $x,w\in V(G)$ be such that $\sigma(w)=u$ and $\sigma(v)=x$. Now, 
			\begin{align*}	
				(M_\sigma P)_{uv}=\sum_{y\in V(G)}(M_\sigma)_{uy}P_{yv}&=P_{wv}\\
				&=\sum_{e,f\in\mathcal{A}(G)}N_{we}S_{ef}N_{vf}\\
				&=\mathop{\sum_{e,f\in\mathcal{A}(G)}}_{t(e)=w, t(f)=v} \frac{1}{\sqrt{\deg w}}\frac{1}{\sqrt{\deg v}}\delta_{e,f^{-1}}\\
				&=\left\{ \begin{array}{ll}
					\frac{1}{\sqrt{\deg w \deg v}} &\mbox{ if }
					wv\in E(G) \\ 
					0 &\textnormal{ otherwise.}
				\end{array}\right. 
			\end{align*} 
			Similarly,
			\[
			(PM_\sigma)_{uv}=\sum_{y\in V(G)}P_{uy}(M_\sigma)_{yv}	=P_{ux}=\left\{ \begin{array}{ll}
				\frac{1}{\sqrt{\deg u \deg x}} &\mbox{ if }
				ux\in E(G) \\ 
				0 &\textnormal{ otherwise.}
			\end{array}\right. 
			\]
			Note that $wv\in E(G)$ if and only if $ux\in E(G)$. Also, $\deg w =\deg u$ and $\deg v=\deg x$. Hence $(M_\sigma P)_{uv}=(PM_\sigma)_{uv}$ for any vertices $u$ and $v$. Therefore  $M_\sigma P=PM_\sigma$.
		\end{proof}
		\begin{lema}
			Let $u$ and $v$ be two vertices of a graph $G$, and let $\sigma$ be an automorphism of $G$. If perfect state transfer occurs between $u$ and $v$, then perfect state transfer occurs between  $\sigma(u)$ and $\sigma(v)$.
		\end{lema}
		\begin{proof}
			If perfect state transfer occurs between $u$ and $v$ at time $\tau$, then by Lemma \ref{pstdef}, $T_\tau(P)\eu=\gamma\ev$ for some $\gamma\in\{-1,1\}$. Let $M_\sigma$ be the permutation matrix corresponding to $\sigma\in\aut(G)$. Note that $M_\sigma$ also commutes with $T_\tau(P)$. Therefore $ T_\tau(P)M_\sigma\eu=M_\sigma T_\tau(P)\eu=\gamma M_\sigma\ev.$  Since $M_\sigma\eu=\textbf{e}_{\sigma(u)}$ and $M_\sigma\ev=\textbf{e}_{\sigma(v)}$, it follows that $T_\tau(P)\textbf{e}_{\sigma(u)}=\gamma\textbf{e}_{\sigma(v)}$ for some $\gamma\in\{-1, 1\}$. This completes the proof.
		\end{proof}
		\begin{lema}\label{autg}
			Let $u$ and $v$ be two vertices of a graph $G$. If perfect state transfer occurs between $u$ and $v$, then $\aut(G)_u=\aut(G)_v$.
		\end{lema}
		\begin{proof}
			Let $\sigma\in\aut(G)_u$. Then $M_\sigma\eu=\eu$. If perfect state transfer occurs between $u$ and $v$ at time $\tau$, then by Lemma \ref{pstdef}, there exists $\gamma\in\{-1,1\}$ such that
			$$\gamma\ev=T_\tau(P)\eu=T_\tau(P)M_\sigma\eu=M_\sigma T_\tau(P)\eu=\gamma M_\sigma\ev.$$
			Therefore $\sigma\in\aut(G)_v$. The converse follows from the fact that, if perfect state transfer occurs from $u$ to $v$ at time $\tau$, then perfect state transfer also occurs from $v$ to $u$ with the same time $\tau$.
		\end{proof}
		Lemma~\ref{autg}  also appears in \cite{circ}. However, our proof is significantly simpler than the one in \cite{circ}.
		\begin{lema}
			Let $f$ be an isomorphism from a graph $G_1$ to an isomorphic graph  $G_2$. Then perfect state transfer occurs between two vertices $u$ and $v$ in $G_1$ if and only if perfect state transfer occurs between $f(u)$ and $f(v)$ in $G_2$. 
		\end{lema}
		\begin{proof}
			Let $P_1$ and $P_2$ be the discriminants of the graphs $G_1$ and $G_2$, respectively.	Let $M_f$ be the permutation matrix associated with an isomorphism $f$ from $G_1$ to $G_2$. By a similar argument as in the proof of Lemma \ref{per1}, we find that $M_f P_1=P_2 M_f$, that is, $P_1=M_f^tP_2M_f$.
			
			Suppose perfect state transfer occurs between vertices $u$ and $v$ in $G_1$ at time $\tau$. By Lemma \ref{pstdef}, $T_\tau(P_1)\eu=\gamma\ev$ for some $\gamma\in\{-1, 1\}$. Substituting $P_1=M_f^tP_2M_f$, we have $T_\tau(M_f^tP_2M_f)\eu=\gamma\ev$. Thus, $M_f^tT_\tau(P_2)M_f\eu=\gamma\ev$. This implies $T_\tau(P_2)\textbf{e}_{f(u)}=\gamma\textbf{e}_{f(v)}$. Therefore, perfect state transfer occurs between
			$f(u)$ and $f(v)$ in $G_2$.

			The converse follows directly from the fact that if $f$ is an isomorphism from $G_1$ to $G_2$, then $f^{-1}$ is an isomorphism from $G_2$ to $G_1$.
		\end{proof}
		\section{Unitary Cayley graphs}
		Let $L:=(\ell_{ij})$ and $M$ be two matrices of size $p\times q$ and $r\times s$, respectively. The \emph{Kronecker product} of $L$ and $M$, denoted $L\otimes M$, is a block matrix of size $pr\times qs$, defined as
		$$L\otimes M=\begin{bmatrix}
			\ell_{11}M & \cdots & \ell_{1q}M\\
			\vdots & \ddots & \vdots\\
			\ell_{p1}M&\cdots&\ell_{pq}M
		\end{bmatrix}. $$
		For square matrices $L$ and $M$, we write $L\congp M$ to mean that $L$ and $M$ are permutation similar, that is, there exists a permutation matrix $Q$ such that $L=Q^tMQ$. 
		\begin{lema}[\cite{kronecker}]\label{kp}
			Let $L$ and $M$ be square matrices. Then the following are true.
			\begin{enumerate}[label=(\roman*)]
				\item\label{kp1} $L\otimes M\congp M\otimes L$.
				\item\label{kp2} If the spectral decomposition of $L$ and $M$ are $L=\sum_{j=1}^{p}\ld_jE_j$ and $M=\sum_{j=1}^{q}\mu_jF_j$, respectively, then the spectral decomposition of $L\otimes M$ is $L\otimes M=\sum_{i=1}^{p}\sum_{j=1}^{q}\ld_i\mu_j(E_i\otimes F_j)$.
			\end{enumerate}	
		\end{lema}
		The \emph{tensor product} of two graphs $G$ and $H$ is the graph $G\otimes H$ with vertex set $V(G)\times V(H)$, and two vertices $(a,b)$ and  $(c,d)$ are adjacent if $a$ is adjacent to $c$ in $G$, and $b$ is adjacent to $d$ in $H$. One can check that $A(G\otimes H)=A(G)\otimes A(H)$, and $G\otimes H$ is isomorphic to $H\otimes G$. Note that two graphs $G$ and $H$ are isomorphic if and only if $A(G)\congp A(H)$. Thus, if $G$ and $H$ are two regular graphs, they are isomorphic if and only if $P(G)\congp P(H)$.
		
		Let $R$ be a ring as in Assumption \ref{as}. It is clear that $R^\times =R_1^\times \times \cdots\times R_s^\times$. Therefore, we have $G_R=G_{R_1}\otimes \cdots\otimes G_{R_s}$. The following result gives the adjacency eigenvalues of $G_R$.
		\begin{theorem}[\cite{ucr}\label{evuc}]
			Let $R$ be a ring as in Assumption \ref{as}. Then, the eigenvalues of the adjacency matrix of $G_R$ are 
			\begin{enumerate}[label=(\roman*)]
				\item $(-1)^{|C|}\frac{|R^\times|}{\prod_{j\in C}(|R_j^\times|/m_j)}$, with multiplicity $\prod_{j\in C}(|R_j^\times|/m_j)$, where $C$ runs over all subsets of $\{1,\hdots,s\}$, and  	
				\item $0$, with multiplicity $|R|-\prod_{j=1}^{s}\left(1+\frac{|R_j^\times|}{m_j} \right)$. 
			\end{enumerate}
		\end{theorem}
		In reference to notations of the previous lemma, let $\ld_C=(-1)^{|C|}\frac{|R^\times|}{\prod_{j\in C}(|R_j^\times|/m_j)}$, where $C$ runs over all subsets of $\{1,\hdots,s\}$. In particular, $\ld_\emptyset=|R^\times|$ and $\ld_{\{1\}}=-m_1|R^\times|/|R_1^\times|$. Note that $G_R$ is a regular graph of degree $|R^\times|$. Thus by Theorem \ref{discp}, the eigenvalues of the discriminant matrix $P$ of $G_R$ are 
		
		$$	\left\{ \begin{array}{rl}
			\mu_C &\textnormal{ with multiplicity   $\prod_{j\in C}(|R_j^\times|/m_j)$, ~and}\\ 
			0 &\textnormal{ with multiplicity $|R|-\prod_{j=1}^{s}\left(1+\frac{|R_j^\times|}{m_j} \right)$,}
		\end{array}\right.$$
		where $\mu_C=\frac{\ld_C}{|R^\times|}$ and $C$ runs over all subsets of $\{1,\hdots,s\}$.
		
		\begin{theorem}\label{ucrp}
			Let $R$ be a ring as in Assumption \ref{as}. Then the unitary Cayley graph $G_R$ is periodic if and only if either $R_1/M_1\cong \Zl_2$ or $R_1/M_1\cong \Zl_3$, and $R_i/M_i\cong \Zl_2$ for $i\in\{2,\hdots,s\}$.
		\end{theorem}
		\begin{proof}
			First, assume that $G_R$ is a periodic graph.  Since $M_i$ is the maximal ideal of the local ring $R_i$, $R_i/M_i$ is a field, and so  $|R_i|/m_i\geq 2$ for $1\leq i\leq s$. Suppose $|R_1|/m_1\geq 4$. Since $R_1$ is a local ring with maximal ideal $M_1$, one can find that $R_1^\times = R_1\setminus M_1$. Therefore $|R_1^\times| /m_1=|R_1|/m_1-1\geq 3$. Note that $|R_1|/m_1$ is an integer, and so $-\frac{1}{3}\leq -m_1/|R_1^\times|< 0$. Consider the subset $\{1\}$ of $\{1,\hdots, s\}$. Then 
			$$\mu_{\{1\}}=-\frac{m_1}{|R_1^\times|}	\notin \left\{\pm 1,\pm\frac{1}{2},0\right\}.$$ 
			Then  Theorem~\ref{ls} gives that $G_R$ is not periodic, a contradiction. Thus  $|R_1|/m_1\leq3$, and so $|R_i|/m_i\leq3$ for $i\in\{2,\hdots, s\}$. Now, assume that $s\geq 2$.  Suppose, if possible, $|R_2|/m_2=3$. Then we must have $|R_1|/m_1= 3$. Consider the subset $\{1,2\}$  of $\{1,2,\hdots, s\}$. We have $\mu_{\{1,2\}}=\frac{1}{4}$, contradiction that $G_R$ is periodic. Hence $2\leq |R_1|/m_1\leq 3$, and $|R_i|/m_i=2$ for $i\in \{2,\hdots, s\}$. Since $R_i/M_i$ is a field for each $i\in\{1,\hdots,s\}$, we are done.

			Conversely assume that  $R_1/M_1\cong \Zl_2$ or $R_1/M_1\cong \Zl_3$, and $R_i/M_i\cong \Zl_2$ for $i\in\{2,\hdots,s\}$.  Then by Theorem~\ref{evuc}, the discriminant eigenvalues of $G_R$ are either $0$ or $\mu_C$, where $\mu_C=\frac{(-1)^{|C|}}{\prod_{j\in C}(|R_j^\times|/m_j)}$. Note that $|R_j^\times|/m_j=|R_j|/m_j-1$ for each $j$. Therefore if $R_i/M_i\cong \Zl_2$ for $i\in\{1,\hdots,s\}$, then 		
			$$\frac{(-1)^{|C|}}{\prod_{j\in C}(|R_j^\times|/m_j)}=\left\{ \begin{array}{rl}
				1 &\textnormal{if $|C|$ is even  }\\ 
				-1 &\textnormal{if $|C|$ is odd.}
			\end{array}\right.$$
			Similarly, if $R_1/M_1\cong \Zl_3$ and $R_i/M_i\cong \Zl_2$ for $i\in\{2,\hdots,s\}$, then 
			$$\frac{(-1)^{|C|}}{\prod_{j\in C}(|R_j^\times|/m_j)}=\left\{ \begin{array}{rl}
				1 &\textnormal{if $1\notin C$ and $|C|$ is even  }\\ 
				-1 &\textnormal{if $1\notin C$ and $|C|$ is odd}\\
				\frac{1}{2}&\textnormal{if $1\in C$ and $|C|$ is even }\\
				-\frac{1}{2}&\textnormal{if $1\in C$ and $|C|$ is odd.}
			\end{array}\right.$$
			Thus $\spec_P(G_R)\subseteq\{\pm1,\pm \frac{1}{2},0\}$. Therefore, the graph $G_R$ is periodic.
		\end{proof}
		We now prove several lemmas to characterize perfect state transfer on $G_R$.
		\begin{lema}\label{perm}
			Let $R$ be a ring as in Assumption \ref{as}. If $G_R$ exhibits perfect state transfer, then $m\leq2$.
		\end{lema}
		\begin{proof}
			Since each $R_i$ is a finite local ring, $R_i^\times=R_i\setminus M_i$ for $i\in \{1,\hdots,s\}$. Thus, two vertices $u$ and $v$ in $G_{R_i}$ are adjacent if and only if $u-v\notin M_i$ for $i\in \{1,\hdots,s\}$. Therefore, the graph $G_{R_i}$ is a complete multipartite graph, where the partite sets correspond to the cosets of $M_i$ in $R_i$ for $i\in \{1,\hdots,s\}$. Hence 
			$$ A(G_{R_i})= J_{m_i} \otimes A(K_{\frac{|R_i|}{m_i}}),$$
			where $i\in \{1,\hdots,s\}$ and $J_{n}$ denotes the all all-ones matrix of size $n$.
			Note that $P(G_R)=\frac{1}{|R^\times|}A(G_R)$ and $P(G_{R_i})=\frac{1}{|R_i^\times|}A(G_{R_i})$ for $i\in\{1,\hdots,s\}$. Since $G_R= G_{R_1}\otimes \cdots \otimes G_{R_s}$, we have
			\begin{align*}
				P(G_R)&= P(G_{R_1})\otimes\cdots\otimes P(G_{R_s}) \\
				&= \bigotimes_{i=1}^s \left( \frac{1}{m_i}J_{m_i}\otimes P({K_{\frac{|R_i|}{m_i}}})\right) \\ 
				&\congp   \bigotimes_{i=1}^n P({K_{\frac{|R_i|}{m_i}}}) \otimes\bigotimes_{i=1}^s \frac{1}{m_i}J_{m_i} \tag*{(by \ref{kp1} of Lemma \ref{kp})}\\
				&=  P(H)\otimes \frac{1}{m}J_m ,
			\end{align*}
			where $H$ is the tensor product $K_{\frac{|R_1|}{m_1}}\otimes \cdots \otimes K_{\frac{|R_s|}{m_s}}$.
			
			Let perfect state transfer occur between distinct vertices $u$ and $v$ at time $\tau$ in $G_R$, and suppose that $m\geq 3$. The spectral decomposition of $\frac{1}{m}J_m$ is 
			$$\frac{1}{m}J_m=1\cdot E_1+0\cdot E_2, $$
			where $E_1=\frac{1}{m}J_m$ and $E_2=I_m-\frac{1}{m}J_m$. Let the spectral decomposition of $P(H)$ be $\sum_{j=1}^{d}\theta_j F_j$. Then by Lemma \ref{kp} \ref{kp2}, the eigenprojection of $P(G_R)$ corresponding to the eigenvalue $0$ is $(F_1+\cdots+F_d)\otimes E_2$, which is equal to $I\otimes E_2$. Let $\ex, \ey,\ez$ and $\ew$ be such that $\eu=\ex\otimes\ey$ and $\ev=\ez\otimes\ew$. By Lemma \ref{pstc}, there exists $\gamma\in\{-1,1\}$ such that 
			\[	(I\otimes E_2)\eu=\gamma(I\otimes E_2)\ev. \]
			This implies
			\begin{equation}\label{co1}
				\ex\otimes E_2\ey=\ez\otimes\gamma E_2\ew.
			\end{equation}	
			If $\ex\neq\ez$, then \eqref{co1} is not possible. If $\ex=\ez$, then \eqref{co1} gives $E_2\ey=\gamma E_2\ew$. Note that $u$ and $v$ are distinct vertices, so we must have $\ey\neq\ew$. As $m\geq 3$ and $E_2=I_m-\frac{1}{m}J_m$, we find that $E_2\ey\neq \gamma E_2\ew$, a contradiction. Thus $m\leq2$.
		\end{proof}
		A ring $R$ is said to be an \emph{$S$-ring} if every element of $R$ can be written as a finite sum of elements of units in $R$. See \cite{sring} for more information about $S$-rings. Note that  a $k$-regular graph is connected if and only if $k$ is an eigenvalue of the adjacency matrix of the graph with multiplicity one. We prove the following result about $S$-rings using spectral properties of unitary Cayley graphs.
		\begin{lema}\label{sring}
			Let $R$ be a ring as in Assumption \ref{as}. Then $R$ is an $S$-ring if and only if $R_j/M_j\cong \Zl_2$ for at most one $j\in\{1,\hdots,s\}$.
		\end{lema}
		\begin{proof}
			Consider the unitary Cayley graph $G_R$. It is clear that $G_R$ is connected if and only if $R$ is an $S$-ring. Therefore, $R$ is an $S$-ring if and only if $|R^\times|$ is an eigenvalue of the adjacency matrix of $G_R$ with multiplicity $1$.

			Let $s=1$. Then by Theorem~\ref{evuc}, $|R^\times|$ is an eigenvalue of $A(G_R)$ with multiplicity $1$. Further, $R_j/M_j\cong \Zl_2$ for at most one $j\in\{1\}$. Thus, the result is proved in this case.

			Now suppose that $s\geq 2$. Let $R$ be an $S$-ring, and suppose that there exist distinct $R_j$ and $R_k$ such that $R_j/M_j\cong \Zl_2$ and $R_k/M_k\cong \Zl_2$. Then by Theorem~\ref{evuc}, we have $\ld_{\emptyset}=|R^\times|$ and $\ld_{\{j,k\}}=|R^\times|$. This leads to the contradiction that $|R^\times|$ is an eigenvalue of $A(G_R)$ with multiplicity at least $2$. 
			
				Conversely, assume that $R_j/M_j \cong \mathbb{Z}_2$ for at most one 
			$j \in \{1,\ldots,s\}$. Thus ${|R_j|}/{m_j}=2$  for at most one 
			$j \in \{1,\ldots,s\}$. By Theorem~\ref{evuc}, we have 
			$\lambda_{\emptyset} = |R^\times|$.
			Let $\emptyset \neq C \subset \{1,\ldots,s\}$. 
			If $\lambda_C = |R^\times|$, then it follows that $|C|$ is even and 
			$|R_j^\times|/m_j = 1$ for each  $j \in C$. This implies that ${|R_j|}/{m_j}=2$ for each $j\in C$. This contradicts our initial assumption. Hence $|R^\times|$ is an eigenvalue of $A({G}_R)$ with multiplicity one.
		\end{proof}
		The following result completely determines the cyclic group $\Zl_n$ for which $G_{\Zl_n}$ exhibits perfect state transfer.
		\begin{theorem}[\cite{bhakta1}]\label{bh1}
			The unitary Cayley graph $G_{\Zl_n}$ exhibits perfect state transfer if and only if $n\in\{2,4,6,12\}$. 
		\end{theorem}
		We now extend this result (Theorem \ref{bh1}) on the occurrence of perfect state transfer on the unitary Cayley graph over $\Zl_n$ to any finite commutative ring $R$.
		\begin{theorem}\label{ucrpst}
			Let $R$ be a ring as in Assumption \ref{as}. The unitary Cayley graph $G_R$ exhibits perfect state transfer if and only if $R$ is one of the following rings: $\Zl_2$, $\Zl_4$, $\mathbb{F}_2[x]/(x^2)$, $\Zl_6$,  $\Zl_{12}$, or $\Zl_3 \times \mathbb{F}_2[x]/(x^2)$.
		\end{theorem}
		
		\begin{proof}
			It is known that a Cayley graph is a vertex-transitive graph. Therefore by Theorem~\ref{thmper} and Theorem~\ref{ucrp}, we conclude that $R_1/M_1\cong \Zl_2$ or $R_1/M_1\cong \Zl_3$, and $R_i/M_i\cong \Zl_2$ for $i\in\{2,\hdots,s\}$. Note that perfect state transfer cannot occur between vertices of different components. Thus, without loss of generality, we consider the graph $G_R$ to be connected. Hence by Lemma \ref{sring}, we have
			\begin{enumerate}[label=(\roman*)]
				\item $R=R_1$, where $R_1/M_1\cong \Zl_2$ or $R_1/M_1\cong \Zl_3$; or
				\item $R=R_1\times R_2$, where $R_1/M_1\cong \Zl_3$ and $R_2/M_2\cong \Zl_2$.
			\end{enumerate}	
			\begin{figure}[h!]
				\centering
				\begin{subfigure}{.5\textwidth}
					\centering

					\begin{tikzpicture}[scale=.9] 
						
						\node[circle, draw, fill=black, inner sep=2pt] (Na) at ({360/4*0}:2) {};
						\node[circle, draw, fill=black, inner sep=2pt] (Nb) at ({360/4*1}:2) {};
						\node[circle, draw, fill=black, inner sep=2pt] (Nc) at ({360/4*2}:2) {};
						\node[circle, draw, fill=black, inner sep=2pt] (Nd) at ({360/4*3}:2) {};

						\node[right] at (Na) {$1$};
						\node[below, xshift=.2pt, yshift=15pt] at (Nb) {$0$}; 
						\node[left] at (Nc) {$3$};
						\node[below, xshift=.2pt, yshift=-2pt] at (Nd) {$2$}; 
						
						\node[below, xshift=.2pt, yshift=-25pt] at (Nd) {$\Zl_4^\times=\{1,3\}$}; 
						Draw cycle edges
						\draw (Na) -- (Nb);
						\draw (Nb) -- (Nc);
						\draw (Nc) -- (Nd);
						\draw (Nd) -- (Na);
						
						
					\end{tikzpicture}
					\caption{$G_{\mathbb{Z}_{4}}$}\label{fig:sub1}
				\end{subfigure}%
				\begin{subfigure}{.5\textwidth}
					\centering

					\begin{tikzpicture}[scale=1] 
						
						\node[circle, draw, fill=black, inner sep=2pt] (Na) at ({360/4*0}:2) {};
						\node[circle, draw, fill=black, inner sep=2pt] (Nb) at ({360/4*1}:2) {};
						\node[circle, draw, fill=black, inner sep=2pt] (Nc) at ({360/4*2}:2) {};
						\node[circle, draw, fill=black, inner sep=2pt] (Nd) at ({360/4*3}:2) {};

						\node[right] at (Na) {$1$};
						\node[below, xshift=.2pt, yshift=15pt] at (Nb) {$0$}; 
						\node[left, xshift=-2pt, yshift=0pt] at (Nc) {$1+x$};
						\node[below, xshift=.2pt, yshift=-2pt] at (Nd) {$x$}; 
						\node[below, xshift=.2pt, yshift=-25pt] at (Nd) {$	\left(\frac{\mathbb{F}_2[x]}{(x^2)}\right)^\times=\{1,1+x\}$};
						
						Draw cycle edges
						\draw (Na) -- (Nb);
						\draw (Nb) -- (Nc);
						\draw (Nc) -- (Nd);
						\draw (Nd) -- (Na);

					\end{tikzpicture}			\caption{$G_{\mathbb{F}_2[x]/(x^2)}$ }
					\label{fig:sub2}
				\end{subfigure}
				\caption{The graphs $G_{\Zl_4}$ and $G_{\mathbb{F}_2[x]/(x^2)}$}
				\label{fig1}
			\end{figure}
			By Lemma \ref{perm}, $m\leq 2$. If $m=1$ then the only possible choices of $R$ are $\Zl_2$, $\Zl_3$ and $\Zl_6$. Now consider the case $m=2$.	
			\begin{figure}[h!]
				\centering
				\begin{subfigure}{.5\textwidth}
					\centering

					\begin{tikzpicture}[scale=1.4] 
						
						\node[circle, draw, fill=black, inner sep=2pt] (Na) at ({360/12*0}:2) {};
						\node[circle, draw, fill=black, inner sep=2pt] (Nb) at ({360/12*1}:2) {};
						\node[circle, draw, fill=black, inner sep=2pt] (Nc) at ({360/12*2}:2) {};
						\node[circle, draw, fill=black, inner sep=2pt] (Nd) at ({360/12*3}:2) {};
						\node[circle, draw, fill=black, inner sep=2pt] (Ne) at ({360/12*4}:2) {};
						\node[circle, draw, fill=black, inner sep=2pt] (Nf) at ({360/12*5}:2) {};
						\node[circle, draw, fill=black, inner sep=2pt] (Ng) at ({360/12*6}:2) {};
						\node[circle, draw, fill=black, inner sep=2pt] (Nh) at ({360/12*7}:2) {};
						\node[circle, draw, fill=black, inner sep=2pt] (Ni) at ({360/12*8}:2) {};
						\node[circle, draw, fill=black, inner sep=2pt] (Nj) at ({360/12*9}:2) {};
						\node[circle, draw, fill=black, inner sep=2pt] (Nk) at ({360/12*10}:2) {};
						\node[circle, draw, fill=black, inner sep=2pt] (Nl) at ({360/12*11}:2) {};
						
						\node[right] at (Na) {$3$};
						\node[above right] at (Nb) {$2$};
						\node[above right] at (Nc) {$1$};
						\node[above right, xshift=-6pt, yshift=2pt] at (Nd) {$0$}; 
						\node[above left] at (Ne) {$11$};
						\node[above left] at (Nf) {$10$};
						\node[left] at (Ng) {$9$};
						\node[below left] at (Nh) {$8$};
						\node[below left] at (Ni) {$7$};
						\node[above right, xshift=-6pt, yshift=-15pt] at (Nj) {$6$}; 
						\node[below right] at (Nk) {$5$};
						\node[below right] at (Nl) {$4$};
						
						\node[above right, xshift=-40pt, yshift=-50pt] at (Nj) {$\Zl_{12}^\times=\{1,5,7,11\}$}; 
						
						\draw (Na) -- (Nb);
						\draw (Nb) -- (Nc);
						\draw (Nc) -- (Nd);
						\draw (Nd) -- (Ne);
						\draw (Ne) -- (Nf);
						\draw (Nf) -- (Ng);
						\draw (Ng) -- (Nh);
						\draw (Nh) -- (Ni);
						\draw (Ni) -- (Nj);
						\draw (Nj) -- (Nk);
						\draw (Nk) -- (Nl);
						\draw (Nl) -- (Na);
						\draw (Nd) -- (Ni);
						\draw (Nd) -- (Nk);
						\draw (Nc) -- (Nj);
						\draw (Nc) -- (Nh);
						\draw (Nb) -- (Ni);
						\draw (Nb) -- (Ng);
						\draw (Na) -- (Nh);
						\draw (Na) -- (Nf);
						\draw (Nl) -- (Ng);
						\draw (Nl) -- (Ne);
						\draw (Nk) -- (Nf);
						\draw (Ne) -- (Nj);
						
					\end{tikzpicture}
					\caption{$G_{\mathbb{Z}_{12}}$}
				\end{subfigure}%
				\begin{subfigure}{.5\textwidth}
					\centering

					
					\begin{tikzpicture}[scale=1.4] 
						
						\node[circle, draw, fill=black, inner sep=2pt] (Na) at ({360/12*0}:2) {};
						\node[circle, draw, fill=black, inner sep=2pt] (Nb) at ({360/12*1}:2) {};
						\node[circle, draw, fill=black, inner sep=2pt] (Nc) at ({360/12*2}:2) {};
						\node[circle, draw, fill=black, inner sep=2pt] (Nd) at ({360/12*3}:2) {};
						\node[circle, draw, fill=black, inner sep=2pt] (Ne) at ({360/12*4}:2) {};
						\node[circle, draw, fill=black, inner sep=2pt] (Nf) at ({360/12*5}:2) {};
						\node[circle, draw, fill=black, inner sep=2pt] (Ng) at ({360/12*6}:2) {};
						\node[circle, draw, fill=black, inner sep=2pt] (Nh) at ({360/12*7}:2) {};
						\node[circle, draw, fill=black, inner sep=2pt] (Ni) at ({360/12*8}:2) {};
						\node[circle, draw, fill=black, inner sep=2pt] (Nj) at ({360/12*9}:2) {};
						\node[circle, draw, fill=black, inner sep=2pt] (Nk) at ({360/12*10}:2) {};
						\node[circle, draw, fill=black, inner sep=2pt] (Nl) at ({360/12*11}:2) {};
						
						\node[right] at (Na) {$(0,1+x)$};
						\node[above right] at (Nb) {$(2,x)$};
						\node[above right] at (Nc) {$(1,1+x)$};
						\node[above right, xshift=-14pt, yshift=2pt] at (Nd) {$(0,0)$}; 
						\node[above left] at (Ne) {$(2,1+x)$};
						\node[above left] at (Nf) {$(1,0)$};
						\node[left] at (Ng) {$(0,1)$};
						\node[below left] at (Nh) {$(2,0)$};
						\node[below left] at (Ni) {$(1,1)$};
						\node[above right, xshift=-14pt, yshift=-17pt] at (Nj) {$(0,x)$}; 
						\node[below right] at (Nk) {$(2,1)$};
						\node[below right] at (Nl) {$(1,x)$};
						
						\node[above right, xshift=-115pt, yshift=-50pt] at (Nj) {$\left(\mathbb{Z}_3\times \frac{\mathbb{F}_2[x]}{(x^2)}\right)^\times=\{(1,1), (2,1), (1,1+x),(2,1+x)\}$}; 
						
						\draw (Na) -- (Nb);
						\draw (Nb) -- (Nc);
						\draw (Nc) -- (Nd);
						\draw (Nd) -- (Ne);
						\draw (Ne) -- (Nf);
						\draw (Nf) -- (Ng);
						\draw (Ng) -- (Nh);
						\draw (Nh) -- (Ni);
						\draw (Ni) -- (Nj);
						\draw (Nj) -- (Nk);
						\draw (Nk) -- (Nl);
						\draw (Nl) -- (Na);
						\draw (Nd) -- (Ni);
						\draw (Nd) -- (Nk);
						\draw (Nc) -- (Nj);
						\draw (Nc) -- (Nh);
						\draw (Nb) -- (Ni);
						\draw (Nb) -- (Ng);
						\draw (Na) -- (Nh);
						\draw (Na) -- (Nf);
						\draw (Nl) -- (Ng);
						\draw (Nl) -- (Ne);
						\draw (Nk) -- (Nf);
						\draw (Ne) -- (Nj);
						
					\end{tikzpicture}
					\caption{$G_{\mathbb{Z}_3\times \mathbb{F}_2[x]/(x^2)}$}
				\end{subfigure}
				\caption{The graphs $G_{\Zl_{12}}$ and $G_{\Zl_{3}\times \mathbb{F}_2[x]/(x^2)}$ }
				\label{fig2}
			\end{figure}

			\noindent\textbf{Case 1}. $R=R_1$, where $R_1/M_1\cong \Zl_2$ or $R_1/M_1\cong \Zl_3$. In this case, $|R_1|=4$ or $|R_1|=6$. 
			
			From Theorem \ref{clr}, we find that there are only $4$ rings of order $4$ with unity, namely, $\Zl_4$, $\Zl_2\times\Zl_2$, $\mathscr{G}(2)$ and $\mathbb{F}_4$. Recall that $\mathscr{G}(2)\cong \mathbb{F}_2[x]/(x^2)$. However, $\Zl_2\times \Zl_2$ is not a local ring, and $m=1$ for $\mathbb{F}_4$. We thus have $R=\Zl_4$ or $R=\mathbb{F}_2[x]/(x^2)$.
			
			The rings of order $6$ are $\Zl_6$, $C_6(0)$, $C_2(0)\times\Zl_3$ and $C_3(0)\times\Zl_2$. However, $\Zl_6$ is not a local ring, and the rings $C_6(0)$, $C_2(0)\times\Zl_3$ and $C_3(0)\times\Zl_2$ do not contain unity. Therefore, $|R_1|=6$ is not possible.
			
			\noindent\textbf{Case 2}. $R=R_1\times R_2$, where $R_1/M_1\cong \Zl_3$ and $R_2/M_2\cong \Zl_2$. Choosing  $m_1=2$ and $m_2=1$ gives $|R_1|=6$. As in the previous case, we find that $|R_1|=6$ is not possible. Therefore, we must have $m_1=1$ and $m_2=2$. This implies $R_1=\Zl_3$, and $R_2=\Zl_4$ or $R_2=\mathbb{F}_2[x]/(x^2)$. Consequently, $R=\Zl_{12}$ or $R=\Zl_3\times\mathbb{F}_2[x]/(x^2)$.
			
			Thus if $G_R$ exhibits perfect state transfer, then the possible choices of $R$ are $\Zl_2$, $\Zl_3$, $\Zl_4$, $\mathbb{F}_2[x]/(x^2)$, $\Zl_6$, $\Zl_{12}$ and  $\Zl_3 \times \mathbb{F}_2[x]/(x^2)$.
			One can check that the graphs $G_{\Zl_4}$ and $G_{\mathbb{F}_2[x]/(x^2)}$ are isomorphic. Similarly, the graphs $G_{\Zl_{12}}$ and $G_{\Zl_3 \times \mathbb{F}_2[x]/(x^2)}$ are also isomorphic. For example, see Figures~\ref{fig1}~and~\ref{fig2}. 
			
			By Theorem \ref{bh1}, we find that   the only unitary Cayley graphs $G_{\Zn}$ exhibiting perfect state transfer are $G_{\mathbb{Z}_2}$, $G_{\mathbb{Z}_4}$, $G_{\mathbb{Z}_6}$ and $G_{\mathbb{Z}_{12}}$. Hence, we have the desired result.
		\end{proof}

		\section{Quadratic unitary Cayley graphs}
		The \emph{complete pseudograph}, denoted $\mathring{K_n}$, is obtained by attaching a loop to each vertex of $K_n$. The following result is helpful to find the discriminant matrix of $\Gr$ for a local ring $R$.
		\begin{lema}[\cite{qucr}]\label{zad}
			Let $R$ be a finite local ring with unique maximal ideal $M$. If $|R|/|M|$ is odd, then $\Gr\cong \mathcal{G}_{R/M} \otimes \mathring{K}_{|M|}$.
		\end{lema}
		For two disjoint subsets $B$ and $C$ of $\{1,\hdots,s\}$, define
		\begin{equation*}\label{expf}
			\ld_{B,C}=(-1)^{|C|}\dfrac{|R^\times|}{2^s\prod_{i\in B}\left(\sqrt{|R_i|/m_i}+1\right)\prod_{j\in C}\left(\sqrt{|R_j|/m_j}-1\right)}.\end{equation*}
		The following result gives the eigenvalues of the adjacency matrix of $\Gr$.
		\begin{theorem}[\cite{qucr}]\label{evqucr}
			Let $R$ be a ring as in Assumption \ref{as} such that  $|R_i|/m_i\equiv 1 \pmod 4$ for \linebreak[4] $i\in\{1,\hdots,s\}$. Then the eigenvalues of the adjacency matrix of $\Gr$ are
			\begin{enumerate}[label=(\roman*)]
				\item $\ld_{B,C}$, with multiplicity $\dfrac{1}{2^{|B|+|C|}}\prod\limits_{k\in B\cup C}(|R_k|/m_k-1)$ for all pairs of disjoint subsets $B$ and  $C$ of $\{1,\ldots,s\}$, and
				\item $0$, with multiplicity $|R|-\sum\limits_{\substack{B, C\subseteq\{1,\ldots,s\}\\ B\cap C=\emptyset}}\left(\dfrac{1}{2^{|B|+|C|}}\prod\limits_{k\in B\cup C}(|R_k|/m_k-1)\right)$.
			\end{enumerate}
		\end{theorem}
		The next two lemmas are useful for determining the valency of $\Gr$.
		\begin{lema}\label{field1}
			Let $R$ be a finite local ring with maximal ideal $M$ and unity $1$. Then $|R|/|M| \equiv 1 \pmod 4$ if and only if $-1\in Q_R.$ 	
		\end{lema}
		\begin{proof}
			
			Let $|R|/|M|=q$ and $ q\equiv 1 \pmod 4$. Since $(R/M)^\times$ is a cyclic group, we write $(R/M)^\times = \langle g \rangle$ for some $g \in (R/M)^\times$. Note that $-1 \in (R/M)^\times$, so $-1 = g^a$ for some positive integer $a$. Thus $g^{2a} = 1$. Since $g^{q-1} = 1$, it follows that \(q-1 \mid 2a\). Now $q \equiv 1 \pmod{4}$, so $a$ must be even. Let $a = 2b$ for some positive integer $b$. Substituting this in $g^a=-1$, we find $g^{2b} = -1$. This implies that $(g^b)^2 = -1$. Note that $g^b \in (R/M)^\times$. Therefore $-1 \in Q_{R/M}$. By Hensel's Lemma \cite{ict} for local rings, it follows that $-1 \in Q_R$.  
			
			Conversely, assume that $-1\in Q_R$. Then one can check that $-1\in Q_{R/M}$. Thus $-1=g^2$ for some $g\in (R/M)^\times$. Thus the order of $g$ in the group $(R/M)^\times$ is $4$. This implies that $4$ divides $|R|/|M|-1$. Hence $|R|/|M| \equiv 1 \pmod 4$.
		\end{proof}
		\begin{lema}\label{szz}
			Let $p$ be an odd prime and $q=p^n$, where $n$ is a positive integer. 	Let $\mathbb{F}_q$ be the finite field of order $q$. Then $|Q_{\mathbb{F}_q}|=\dfrac{q-1}{2}$.
		\end{lema}
		\begin{proof}
			Consider the map \(f:{\mathbb{F}_q}^\times\rightarrow Q_{\mathbb{F}_q}\) defined by $f(a)=a^2$. Clearly, $f$ is a surjective homomorphism. Note that ${\mathbb{F}_q}^\times$ is a cyclic group with $|{\mathbb{F}_q}^\times|=q-1$, which is an even number. This altogether implies $|\textnormal{Ker}(f)|=2$. Therefore $|Q_{\mathbb{F}_q}|=\dfrac{|{\mathbb{F}_q}^\times|}{2}=\dfrac{q-1}{2}$.
		\end{proof}
		Let $R$ be a finite local ring with maximal ideal $M$ such that $|R|/|M|=1 \pmod 4$. As $R/M$ is a field, by Lemma~\ref{szz}, we have
		$$|Q_{R/M}|=\dfrac{|R|/|M|-1}{2}=\dfrac{|R^\times|}{2|M|}.$$
		It was proved in \cite{qucr} that if $R$ is a finite local ring with maximal ideal $M$, then $Q_R\cong Q_{R/M}\times (1+M)$ as multiplicative groups. Let $R$ be a ring as in Assumption \ref{as} such that $|R_i|/m_i\equiv 1 \pmod 4$ for $i\in\{1,\hdots,s\}$. Then by Lemma~\ref{field1}, $-1\in Q_{R_i}$ for $i\in\{1,\hdots,s\}$. Note that $Q_R=Q_{R_1}\times\cdots \times Q_{R_s}$. Thus $-1\in Q_R$, and so $T_R=Q_R$. Therefore the valency $k$ of the graph $\Gr$ is given by
		$$k=|Q_R|= \prod_{i=1}^{s}|Q_{R_i}|=\prod_{i=1}^{s} |Q_{R_i/M_i}||M_i|= \frac{|R^\times|}{2^s}.$$
		For two disjoint subsets $B$ and $C$ of $\{1,\hdots,s\}$, let $$\mu_{B,C}=(-1)^{|C|}\dfrac{1}{\prod_{i\in B}\left(\sqrt{|R_i|/m_i}+1\right)\prod_{j\in C}\left(\sqrt{|R_j|/m_j}-1\right)}.$$
		
		Then by Theorem \ref{discp}, the eigenvalues of $P(\Gr)$ are $0$ and $\mu_{B,C}$, where $B$ and $C$ are disjoint subsets of $\{1,\hdots,s\}$. Their respective multiplicities are provided in Theorem \ref{evqucr}.
		\begin{theorem}\label{qucrp}
			Let $R$ be a ring as in Assumption \ref{as} such that $|R_i|/m_i\equiv 1 \pmod 4$ for $i\in\{1,\hdots,s\}$. Then the quadratic unitary Cayley graph $\Gr$ is periodic if and only if $s=1$ and $R_1/M_1\cong \mathbb{Z}_5$.
		\end{theorem}
		\begin{proof}
			Let $\Gr$ be periodic. Consider the subsets $B=\{i\}$ and $C=\emptyset$ of $\{1,\hdots,s\}$, where $i\in\{1,\hdots,s\}$. We have
			$$\mu_{B,C}=\dfrac{\sqrt{|R_i|/m_i}-1}{|R_i|/m_i-1}.$$
			By Theorem \ref{ls}, we must have 
			$$\dfrac{\sqrt{|R_i|/m_i}-1}{|R_i|/m_i-1}\in \left\{\pm \frac{\sqrt{3}}{2}, \pm\frac{1}{4}\pm\frac{\sqrt{5}}{4},\pm\frac{1}{\sqrt{2}}\right\}~\text{for}~i\in\{1,\hdots,s\}.$$
			Consequently $|R_i|/m_i=5$ for each $i$. Since $R_i/M_i$ is a field, we finally have $R_i/M_i\cong\Zl_5$ for $i\in\{1,\hdots,s\}$. 
			
			Suppose, if possible, $s\geq 2$. Then for the subsets $B=\{1\}$ and $C=\{2\}$, we have
			$$\mu_{B,C}=-\frac{1}{4}\notin  \left\{\pm1,\pm\frac{1}{2},0 \right\}.$$
			Theorem \ref{ls} again gives that $\Gr$ is not periodic, a contradiction. Thus $s=1$ and $R_1/M_1\cong \Zl_5$.
			
			Conversely, suppose that $s=1$ and $R_1/M_1\cong \Zl_5$. Then the nonzero eigenvalues of $P(\Gr)$ are $1,\frac{\sqrt{5}-1}{4}$ and  $\frac{-\sqrt{5}-1}{4}$. Therefore, by Theorem \ref{ls}, $\Gr$ is periodic.
		\end{proof}
		Indeed, if $R$ is a finite local ring with the maximal ideal $M$ such that $R/M \cong \mathbb{Z}_5$ and $m=|M|$, then $\mathcal{G}_R \cong C_5[\overline{K}_m]$, the lexicographic product of the cycle on $5$ vertices with the complement of the complete graph on $m$ vertices.  This can be seen as follows. 
		
		Consider the surjective homomorphism $\pi : R \to R/M$ defined by $\pi(r)=r+M$ for $r\in R$. Fix a ring isomorphism $f:R/M\to \Zl_5$ and identify the element $f^{-1}(i)$ of $R/M$ with $i$ itself, for each $i\in\Zl_5$. Thus $\pi(r)\in\{0,1,2,3,4\}$ for $r\in R$. 
		
		We now prove that $Q_R = \{r\in R:\pi(r)\in\{1,4\}\}$. If $r \in R^\times$, then $\pi(r) \in (R/M)^\times= \{1,2,3,4\}$.  
		Hence $\pi(r^2) \in \{1,4\}$.  Thus
		$Q_R \subseteq \{ r \in R : \pi(r) \in \{1,4\}\}$. Conversely,  let $a \in R$ with $\pi(a) \in \{1,4\}$.  Then one can prove that $a$ is a unit and there exists $b_0\in R$ such that $\pi(a)=\pi(b_0)^2$. Since $\pi(b_0)^2\in\{1,4\}$, $b_0$ is also a unit. Consider the polynomial $p(x)=x^2-a\in R[x]$, and let $\bar p(x)\in R/M[x]$ denote its reduction modulo  $M$.  Then $\bar p(b_0+M)=0+M$ and $\bar p'(b_0+M)=2b_0+M$, where $\bar p'(x)=(2+M)x$. Note that $2b_0$ is also a unit in $R$, and so $\bar p'(b_0+M)\neq 0+M$.  By Hensel's lemma \cite{ict} for finite local ring, there exists a simple root $b \in R$ of $p$, that is,  $b^2 = a$.  Moreover, $b$ is a unit.
		Consequently $a = b^2 \in Q_R$.  Hence $\{ r \in R : \pi(r) \in \{1,4\}\} \subseteq Q_R.$ Thus $Q_R=\{r\in R:\pi(r)\in\{1,4\}\}$. 
		
		For each 
		$i \in \{0,1,2,3,4\}$, choose an element $a_i \in R$ such that $
		\pi(a_i) = i$. Consequently, $R$ decomposes as the disjoint union
		\[
		R = \bigcup_{i=0}^{4} (a_i + M).
		\]
		
		Note that $T_R=Q_R$. Let $u\in a_i+M$ and $v\in a_j+M$. Then $u$ and $v$ are adjacent in $\Gr$ if and only if $\pi(u-v)\in\{1,4\}$, that is, if and only if $j-i\equiv \pm1 \pmod{5}.$ For each $i\in\{0,1,2,3,4\}$, identify the elements in  $a_i+M$ with that in $\{i\}\times\{1,\hdots,m\}$.  Thus the elements of $R$ are identified with the elements in $\mathbb{Z}_5\times\{1,\dots,m\}$. With this identification, if $u=(i,\alpha), v=(j,\beta)\in R$ then $u$ is adjacent to $v$ in $\mathcal{G}_R$ 
		if and only if
		$j -i\equiv \pm1 \pmod{5}$.  Thus $\mathcal{G}_R \cong C_5[\overline{K}_m]$.

		\begin{theorem}\label{qucrpst}
			Let $R$ be a ring as in Assumption \ref{as} such that $|R_i|/m_i\equiv 1 \pmod 4$ for $i\in\{1,\hdots,s\}$. Then the only quadratic unitary Cayley graph $\Gr$ exhibiting perfect state transfer is $\mathcal{G}_{\Zl_{10}}$.
		\end{theorem}
		\begin{proof}
			Since the quadratic unitary Cayley graph is vertex-transitive, by Theorem \ref{thmper} and Theorem \ref{qucrp}, $R$ is a local ring with unique maximal ideal $M$ such that $R/M\cong \mathbb{Z}_5$. Note that $|Q_{R/M}|=\frac{|R^\times|}{2m}$ and $|Q_R|=\frac{|R^\times|}{2}$, where $m=|M|$. Thus $P(\Gr)=\frac{2}{|R^\times|}A(\Gr)$ and $P(\mathcal{G}_{R/M})=\frac{2m}{|R^\times|}A(\mathcal{G}_{R/M})$. By Lemma~\ref{zad}, the discriminant matrix of $\Gr$ is given by
			$$P(\Gr)=P(\mathcal{G}_{R/M})\otimes \frac{1}{m}J_m.$$
			As in the proof of Lemma \ref{perm}, we conclude that if perfect state transfer occurs in $\Gr$, then $m\leq 2$. If $m=1$ then clearly $R=\Zl_{5}$. If $m=2$ then $|R|=10$.
			The rings of order $10$ are $\Zl_{10}$, $C_{10}(0)$, $C_2(0)\times\Zl_5$ and $C_5(0)\times\Zl_2$. However, the rings $C_{10}(0)$, $C_5(0)\times\Zl_2$ and $C_2(0)\times\Zl_5$ do not contain unity. Thus $R=\Zl_{10 }$.
			The graphs $\mathcal{G}_{\Zl_5}$ and $\mathcal{G}_{\Zl_{10}}$ are cycles on $5$ and $10$ vertices, respectively. It is known \cite{bhakta1} that the cycle on $n$ vertices exhibits perfect state transfer if and only if $n$ is even. Hence, the result follows.		
		\end{proof}
		We now characterize periodicity and perfect state transfer on $\Gr$ under a different set of conditions on the ring $R$. The next result gives the eigenvalues of $A(\Gr)$, where $R$ is a local ring with maximal ideal $M$ and $|R|/|M|\equiv3\pmod 4$.	
		\begin{theorem}[\cite{qucr}]
			Let $R$ be a local ring with maximal ideal $M$ and $|R|/|M|\equiv3\pmod 4$.	Then the eigenvalues of the adjacency  matrix of $\Gr$ are $|R^\times|$, $-\dfrac{|R^\times|}{|R|/|M|-1}$ and $0$, with multiplicities $1$, $|R|/|M|-1$ and $|R|-|R|/|M|$, respectively.
		\end{theorem}
		\begin{assum}\label{ass2}
			{\em     Let $R$ be a finite commutative ring with unity and $s$ be a non-negative integer. Let the decomposition of $R$ be $R_0\times\cdots\times R_s$ such that $|R_0|/m_0\equiv3\pmod 4$, and $|R_i|/m_i\equiv1\pmod4$ for $i\in\{1,\hdots s\}$, where $R_i$ is a local ring with maximal ideal $M_i$ and $m_i=|M_i|$ for $i\in\{0,\hdots,s\}$. Also, let $\mathcal{R}=R_1\times\cdots\times R_s$  for $s\geq 1$. }
		\end{assum}
		For disjoint subsets $B$ and $C$ of $\{1,\hdots,s\}$, define 
		\begin{equation*}\label{expevquc}
			\ld_{B,C}=(-1)^{|C|}\dfrac{|\mathcal{R}^\times|}{2^s\prod_{i\in B}\left(\sqrt{|R_i|/m_i}+1\right)\prod_{j\in C}\left(\sqrt{|R_j|/m_j}-1\right)}.
		\end{equation*}
		
		The following theorem gives the eigenvalues of the adjacency matrix of $\Gr$ under the conditions on $R$ specified in Assumption \ref{ass2}.
		\begin{theorem}[\cite{qucr}]\label{pstoth}
			Let $R$ be a ring as in Assumption \ref{ass2} and $s\geq 1$. Then the eigenvalues of the adjacency matrix of $\Gr$ are
			\begin{enumerate}[label=(\roman*)]
				\item $|R_0^\times|\lambda_{B,C}$, with multiplicity $\dfrac{1}{2^{|B|+|C|}}\prod\limits_{k\in B\cup C}(|R_k|/m_k-1)$ for all pairs of disjoint subsets $B$ and $C$ of $\{1,\ldots,s\}$, 
				\item $-\dfrac{|R_0^\times|}{|R_0|/m_0-1}\lambda_{B,C}$, with multiplicity $\dfrac{|R_0|/m_0-1}{2^{|B|+|C|}}\prod\limits_{k\in B\cup C}(|R_k|/m_k-1)$ for all pairs of disjoint subsets $B$ and $C$ of $\{1,\ldots,s\}$, and
				\item $0$, with multiplicity $|R|-\sum\limits_{\substack{B,C\subseteq\{1,\ldots,s\}\\ B\cap C=\emptyset}}\left(\dfrac{|R_0|/m_0}{2^{|B|+|C|}}\prod\limits_{k\in B\cup C}(|R_k|/m_k-1)\right)$.
			\end{enumerate}
		\end{theorem}
		Let $R$ be a ring as in Assumption \ref{ass2}. Note that if $s=0$ then the eigenvalues of the adjacency matrix of $\Gr$ are $|R_0^\times|$, $-\dfrac{|R_0^\times|}{|R_0|/m_0-1}$ and $0$, with multiplicities $1$, $|R_0|/m_0-1$ and $|R_0|-|R_0|/m_0$, respectively. Now, we calculate the valency $k$ of $\Gr$. By Lemma \ref{field1}, $-1\notin Q_{R_0}$. Thus $Q_R\cap (-Q_R)=\emptyset$. Therefore $$k=2|Q_R|=2\prod_{j=0}^{s}|Q_{R_j/M_j}||M_j|=2\prod_{j=0}^{s}\dfrac{|R_j^\times|}{2|M_j|}|M_j|=|R_0^\times|\dfrac{|\mathcal{R}^\times|}{2^s}.$$
		For disjoint subsets $B$ and $C$ of $\{1,\hdots,s\}$, define 
		$$\mu_{B,C}=(-1)^{|C|}\dfrac{1}{\prod_{i\in B}\left(\sqrt{|R_i|/m_i}+1\right)\prod_{j\in C}\left(\sqrt{|R_j|/m_j}-1\right)}.
		$$
		Then the nonzero eigenvalues of $P(\Gr)$ are $\mu_{B,C}$ and $-\dfrac{\mu_{B,C}}{|R_0|/m_0-1}$. 
		
		\begin{theorem}\label{peroth}
			Let $R$ be a ring as in Assumption \ref{ass2}. Then the quadratic unitary Cayley graph $\Gr$ is periodic if and only if $s=0$ and $R_0/M_0\cong \Zl_3$. 
		\end{theorem}
		\begin{proof}
			Let $\Gr$ be a periodic graph. Note that $\mu_{\emptyset,\emptyset}=1$. Thus  $-\dfrac{1}{|R_0|/m_0-1}$ is an eigenvalue of $P(\Gr)$. Therefore by Theorem \ref{ls}, if $\Gr$ is periodic then $|R_0|/m_0=3$, and so $R_0/M_0\cong \Zl_3$. Now, suppose that $s\geq 1$. Then for $B=\{1\}$ and $C=\emptyset$, we have $\mu_{B,C}=\frac{\sqrt{|R_1|/m_1}-1}{|R_1|/m_1-1}.$ Therefore by Theorem \ref{ls}, $|R_1|/m_1=5$, that is, $\mu_{B,C}=\frac{\sqrt{5}-1}{4}$. In that case,  $$-\frac{\mu_{B,C}}{|R_0|/m_0-1}=-\dfrac{\sqrt{5}-1}{8}\notin  \left\{\pm \frac{\sqrt{3}}{2}, \pm\frac{1}{4}\pm\frac{\sqrt{5}}{4},\pm\frac{1}{\sqrt{2}}\right\},$$
			a contradiction. Hence $s=0$.

			Conversely, if  $s=0$ and $R_0/M_0\cong \Zl_3$, then $\spec_P(\Gr)\subset\{1,-\frac{1}{2},0\}$. Therefore, Theorem \ref{ls} guarantees  the periodicity of $\Gr$.
		\end{proof}
	Proceeding as in the discussion that followed after the proof of Theorem~\ref{qucrp}, one can see that the graph $\Gr$ in Theorem~\ref{peroth} is isomorphic to $C_3[\overline{K}_m]$, where $m=|M_0|$.
			\begin{theorem}\label{lastth}
	Let $R$ be a ring as in Assumption \ref{ass2}. Then the quadratic unitary Cayley graph $\Gr$ does not exhibit perfect state transfer.
\end{theorem}
\begin{proof}
	Suppose that $\Gr$ exhibit perfect state transfer. Then by Theorem \ref{thmper} and Theorem \ref{peroth}, we have   $s=0$ and $R_0/M_0\cong \Zl_3$. Thus by Lemma \ref{zad}, $P(\mathcal{G}_{R_0})=P(\mathcal{G}_{R_0/M_0})\otimes \frac{1}{m_0}J_{m_0}.$ Now as in the proof of Lemma \ref{perm}, we conclude that $m_0\leq 2$. Hence either $R_0=\Zl_3$ or $R_0=\Zl_6$. However $\Zl_6$ is not a local ring, and therefore $R_0=\Zl_3$. The graph $\mathcal{G}_{\mathbb{Z}_3}$ is the cycle on three vertices. Recall that perfect state transfer occurs in a cycle if and only if the cycle has an even number of vertices. Therefore, the result follows.
\end{proof}	
Note that the ring $\Zl_6$ does not satisfy the hypotheses of Theorems~\ref{qucrpst} and \ref{lastth}. However, being a cycle on six vertices, the graph $\mathcal{G}_{\Zl_6}$ exhibits perfect state transfer. Thus $\mathcal{G}_{\Zl_6}$ is a graph which is not characterized by Theorems~\ref{qucrpst} and \ref{lastth}. It is an interesting problem to determine the adjacency spectrum and consequently characterize periodicity and perfect state transfer of graphs that do not satisfy the hypotheses of Theorems~\ref{qucrpst} and \ref{lastth}.

		\section*{Acknowledgments}
		The first author acknowledges the support provided by the Prime Minister’s Research Fellowship (PMRF) scheme of the Government of India (PMRF-ID: 1903298). The second author thanks the Science and Engineering Research Board (SERB), Government of India, for supporting this work under the MATRICS Project [File No.  MTR/2021/000075]. The authors sincerely thank the anonymous reviewer for their careful reading and constructive comments, which have greatly improved the manuscript.
				
	\end{document}